\documentclass[12pt]{article}
\usepackage[utf8]{inputenc}
\usepackage[T1]{fontenc}

\usepackage{amsmath}
\usepackage{amsfonts}
\usepackage{latexsym}
\usepackage[final]{graphicx}  
\usepackage{amssymb}
\usepackage{amsthm}
\usepackage[margin=2cm]{geometry}
\usepackage{url}
\usepackage{color}
\usepackage{enumerate}
\usepackage[shortlabels]{enumitem} 
\usepackage[small]{caption}
\usepackage[noadjust]{cite}
\usepackage{framed}
\usepackage{ifdraft}
\usepackage[affil-it]{authblk}
\usepackage{tikz}
\usetikzlibrary{arrows,shapes}
\usepackage{caption}
\usepackage{subcaption}
\usepackage{array,calc}

\ifdraft{
    \usepackage[colorinlistoftodos]{todonotes}
    \usepackage{xcolor}
    \usepackage[mark]{gitinfo2}
    
}{
    \usepackage[disable]{todonotes}
}

\usepackage[draft=false]{hyperref} 
\usepackage{cleveref}

\usepackage{pdftexcmds}

\usepackage{xifthen}
\newcommand{\ifequals}[3]{\ifthenelse{\equal{#1}{#2}}{#3}{}}
\newcommand{\case}[2]{#1 #2} 
\newenvironment{switch}[1]{\renewcommand{\case}{\ifequals{#1}}}{}

\newcommand{\letterOrig}[1]{
\begin{switch}{#1}
\case{0}{ \begin{array}{|c|c|c|} \cline{1-3}  \multicolumn{1}{|c}{a} & \multicolumn{1}{c}{b} & \multicolumn{1}{c|}{c} \\  \cline{1-3}  \multicolumn{1}{|c}{c} & \multicolumn{1}{c|}{a} \\  \cline{1-2} \end{array} }
\case{1}{ \begin{array}{|c|c|c|} \cline{2-2}  \multicolumn{1}{c|}{} & a \\  \cline{1-3}  \multicolumn{1}{|c}{c} & \multicolumn{1}{c}{b} & \multicolumn{1}{c|}{a} \\  \cline{1-3} \end{array} }
\case{2}{ \begin{array}{|c|c|c|} \cline{1-2}  \multicolumn{1}{|c}{a} & \multicolumn{1}{c|}{c} \\  \cline{1-3}  \multicolumn{1}{c|}{} & \multicolumn{1}{|c}{c} & \multicolumn{1}{c|}{a} \\  \cline{2-3} \end{array} }
\case{3}{ \begin{array}{|c|c|c|} \cline{1-3}  \multicolumn{1}{|c}{a} & \multicolumn{1}{c}{b} & \multicolumn{1}{c|}{c} \\  \cline{1-3}  \multicolumn{1}{c|}{} & a \\  \cline{2-2} \end{array} }
\case{4}{ \begin{array}{|c|c|c|c|} \cline{2-4}  \multicolumn{1}{c|}{} & \multicolumn{1}{|c}{a} & \multicolumn{1}{c}{b} & \multicolumn{1}{c|}{c} \\  \cline{1-4}  \multicolumn{1}{|c}{c} & \multicolumn{1}{c}{b} & \multicolumn{1}{c|}{a} \\  \cline{1-3} \end{array} }
\case{5}{ \begin{array}{|c|c|c|} \cline{2-3}  \multicolumn{1}{c|}{} & \multicolumn{1}{|c}{a} & \multicolumn{1}{c|}{c} \\  \cline{1-3}  \multicolumn{1}{|c}{c} & \multicolumn{1}{c|}{a} \\  \cline{1-2} \end{array} }
\case{6}{ \begin{array}{|c|c|c|c|} \cline{1-3}  \multicolumn{1}{|c}{a} & \multicolumn{1}{c}{b} & \multicolumn{1}{c|}{c} \\  \cline{1-4}  \multicolumn{1}{c|}{} & \multicolumn{1}{|c}{c} & \multicolumn{1}{c}{b} & \multicolumn{1}{c|}{a} \\  \cline{2-4} \end{array} }
\case{7}{ \begin{array}{|c|c|} \cline{2-2}  \multicolumn{1}{c|}{} & a \\  \cline{1-2}  \multicolumn{1}{|c}{c} & \multicolumn{1}{c|}{a} \\  \cline{1-2} \end{array} }
\case{8}{ \begin{array}{|c|c|} \cline{1-2}  \multicolumn{1}{|c}{a} & \multicolumn{1}{c|}{c} \\  \cline{1-2}  a \\  \cline{1-1} \end{array} }
\case{01}{ \begin{array}{|c|c|c|c|c|} \cline{1-4}  \multicolumn{1}{|c}{a} & \multicolumn{1}{c}{b} & \multicolumn{1}{c|}{c} &  a \\  \cline{1-5}  \multicolumn{1}{|c}{c} & \multicolumn{1}{c|}{a} &  \multicolumn{1}{|c}{c} & \multicolumn{1}{c}{b} & \multicolumn{1}{c|}{a} \\  \cline{1-5} \end{array} }
\case{66}{ \begin{array}{|c|c|c|c|c|c|c|} \cline{1-6}  \multicolumn{1}{|c}{a} & \multicolumn{1}{c}{b} & \multicolumn{1}{c|}{c} &  \multicolumn{1}{|c}{a} & \multicolumn{1}{c}{b} & \multicolumn{1}{c|}{c} \\  \cline{1-7}  \multicolumn{1}{c|}{} & \multicolumn{1}{|c}{c} & \multicolumn{1}{c}{b} & \multicolumn{1}{c|}{a} &  \multicolumn{1}{|c}{c} & \multicolumn{1}{c}{b} & \multicolumn{1}{c|}{a} \\  \cline{2-7} \end{array} }
\case{12}{ \begin{array}{|c|c|c|c|c|} \cline{2-4}  \multicolumn{1}{c|}{} & a &  \multicolumn{1}{|c}{a} & \multicolumn{1}{c|}{c} \\  \cline{1-5}  \multicolumn{1}{|c}{c} & \multicolumn{1}{c}{b} & \multicolumn{1}{c|}{a} &  \multicolumn{1}{|c}{c} & \multicolumn{1}{c|}{a} \\  \cline{1-5} \end{array} }
\case{31}{ \begin{array}{|c|c|c|c|c|} \cline{1-4}  \multicolumn{1}{|c}{a} & \multicolumn{1}{c}{b} & \multicolumn{1}{c|}{c} &  a \\  \cline{1-5}  \multicolumn{1}{c|}{} & a &  \multicolumn{1}{|c}{c} & \multicolumn{1}{c}{b} & \multicolumn{1}{c|}{a} \\  \cline{2-5} \end{array} }
\case{23}{ \begin{array}{|c|c|c|c|c|} \cline{1-5}  \multicolumn{1}{|c}{a} & \multicolumn{1}{c|}{c} &  \multicolumn{1}{|c}{a} & \multicolumn{1}{c}{b} & \multicolumn{1}{c|}{c} \\  \cline{1-5}  \multicolumn{1}{c|}{} & \multicolumn{1}{|c}{c} & \multicolumn{1}{c|}{a} &  a \\  \cline{2-4} \end{array} }
\case{34}{ \begin{array}{|c|c|c|c|c|c|} \cline{1-6}  \multicolumn{1}{|c}{a} & \multicolumn{1}{c}{b} & \multicolumn{1}{c|}{c} &  \multicolumn{1}{|c}{a} & \multicolumn{1}{c}{b} & \multicolumn{1}{c|}{c} \\  \cline{1-6}  \multicolumn{1}{c|}{} & a &  \multicolumn{1}{|c}{c} & \multicolumn{1}{c}{b} & \multicolumn{1}{c|}{a} \\  \cline{2-5} \end{array} }
\case{26}{ \begin{array}{|c|c|c|c|c|c|} \cline{1-5}  \multicolumn{1}{|c}{a} & \multicolumn{1}{c|}{c} &  \multicolumn{1}{|c}{a} & \multicolumn{1}{c}{b} & \multicolumn{1}{c|}{c} \\  \cline{1-6}  \multicolumn{1}{c|}{} & \multicolumn{1}{|c}{c} & \multicolumn{1}{c|}{a} &  \multicolumn{1}{|c}{c} & \multicolumn{1}{c}{b} & \multicolumn{1}{c|}{a} \\  \cline{2-6} \end{array} }
\case{45}{ \begin{array}{|c|c|c|c|c|c|} \cline{2-6}  \multicolumn{1}{c|}{} & \multicolumn{1}{|c}{a} & \multicolumn{1}{c}{b} & \multicolumn{1}{c|}{c} &  \multicolumn{1}{|c}{a} & \multicolumn{1}{c|}{c} \\  \cline{1-6}  \multicolumn{1}{|c}{c} & \multicolumn{1}{c}{b} & \multicolumn{1}{c|}{a} &  \multicolumn{1}{|c}{c} & \multicolumn{1}{c|}{a} \\  \cline{1-5} \end{array} }
\case{62}{ \begin{array}{|c|c|c|c|c|c|} \cline{1-5}  \multicolumn{1}{|c}{a} & \multicolumn{1}{c}{b} & \multicolumn{1}{c|}{c} &  \multicolumn{1}{|c}{a} & \multicolumn{1}{c|}{c} \\  \cline{1-6}  \multicolumn{1}{c|}{} & \multicolumn{1}{|c}{c} & \multicolumn{1}{c}{b} & \multicolumn{1}{c|}{a} &  \multicolumn{1}{|c}{c} & \multicolumn{1}{c|}{a} \\  \cline{2-6} \end{array} }
\case{47}{ \begin{array}{|c|c|c|c|c|} \cline{2-5}  \multicolumn{1}{c|}{} & \multicolumn{1}{|c}{a} & \multicolumn{1}{c}{b} & \multicolumn{1}{c|}{c} &  a \\  \cline{1-5}  \multicolumn{1}{|c}{c} & \multicolumn{1}{c}{b} & \multicolumn{1}{c|}{a} &  \multicolumn{1}{|c}{c} & \multicolumn{1}{c|}{a} \\  \cline{1-5} \end{array} }
\case{84}{ \begin{array}{|c|c|c|c|c|} \cline{1-5}  \multicolumn{1}{|c}{a} & \multicolumn{1}{c|}{c} &  \multicolumn{1}{|c}{a} & \multicolumn{1}{c}{b} & \multicolumn{1}{c|}{c} \\  \cline{1-5}  a &  \multicolumn{1}{|c}{c} & \multicolumn{1}{c}{b} & \multicolumn{1}{c|}{a} \\  \cline{1-4} \end{array} }
\case{78}{ \begin{array}{|c|c|c|c|} \cline{2-4}  \multicolumn{1}{c|}{} & a &  \multicolumn{1}{|c}{a} & \multicolumn{1}{c|}{c} \\  \cline{1-4}  \multicolumn{1}{|c}{c} & \multicolumn{1}{c|}{a} &  a \\  \cline{1-3} \end{array} }
\case{54}{ \begin{array}{|c|c|c|c|c|c|} \cline{2-6}  \multicolumn{1}{c|}{} & \multicolumn{1}{|c}{a} & \multicolumn{1}{c|}{c} &  \multicolumn{1}{|c}{a} & \multicolumn{1}{c}{b} & \multicolumn{1}{c|}{c} \\  \cline{1-6}  \multicolumn{1}{|c}{c} & \multicolumn{1}{c|}{a} &  \multicolumn{1}{|c}{c} & \multicolumn{1}{c}{b} & \multicolumn{1}{c|}{a} \\  \cline{1-5} \end{array} }
\case{41}{ \begin{array}{|c|c|c|c|c|c|} \cline{2-5}  \multicolumn{1}{c|}{} & \multicolumn{1}{|c}{a} & \multicolumn{1}{c}{b} & \multicolumn{1}{c|}{c} &  a \\  \cline{1-6}  \multicolumn{1}{|c}{c} & \multicolumn{1}{c}{b} & \multicolumn{1}{c|}{a} &  \multicolumn{1}{|c}{c} & \multicolumn{1}{c}{b} & \multicolumn{1}{c|}{a} \\  \cline{1-6} \end{array} }
\case{012}{ \begin{array}{|c|c|c|c|c|c|c|} \cline{1-6}  \multicolumn{1}{|c}{a} & \multicolumn{1}{c}{b} & \multicolumn{1}{c|}{c} &  a &  \multicolumn{1}{|c}{a} & \multicolumn{1}{c|}{c} \\  \cline{1-7}  \multicolumn{1}{|c}{c} & \multicolumn{1}{c|}{a} &  \multicolumn{1}{|c}{c} & \multicolumn{1}{c}{b} & \multicolumn{1}{c|}{a} &  \multicolumn{1}{|c}{c} & \multicolumn{1}{c|}{a} \\  \cline{1-7} \end{array} }
\case{0123454}{ \begin{array}{|c|c|c|c|c|c|c|c|c|c|c|c|c|c|c|c|c|} \cline{1-17}  \multicolumn{1}{|c}{a} & \multicolumn{1}{c}{b} & \multicolumn{1}{c|}{c} &  a &  \multicolumn{1}{|c}{a} & \multicolumn{1}{c|}{c} &  \multicolumn{1}{|c}{a} & \multicolumn{1}{c}{b} & \multicolumn{1}{c|}{c} &  \multicolumn{1}{|c}{a} & \multicolumn{1}{c}{b} & \multicolumn{1}{c|}{c} &  \multicolumn{1}{|c}{a} & \multicolumn{1}{c|}{c} &  \multicolumn{1}{|c}{a} & \multicolumn{1}{c}{b} & \multicolumn{1}{c|}{c} \\  \cline{1-17}  \multicolumn{1}{|c}{c} & \multicolumn{1}{c|}{a} &  \multicolumn{1}{|c}{c} & \multicolumn{1}{c}{b} & \multicolumn{1}{c|}{a} &  \multicolumn{1}{|c}{c} & \multicolumn{1}{c|}{a} &  a &  \multicolumn{1}{|c}{c} & \multicolumn{1}{c}{b} & \multicolumn{1}{c|}{a} &  \multicolumn{1}{|c}{c} & \multicolumn{1}{c|}{a} &  \multicolumn{1}{|c}{c} & \multicolumn{1}{c}{b} & \multicolumn{1}{c|}{a} \\  \cline{1-16} \end{array} }
\end{switch}}

\newcommand{\letterzuj}[1]{
\begin{switch}{#1}
\case{0}{ \begin{array}{|c|c|c|c|} \cline{1-4}  \multicolumn{1}{|c}{a} & \multicolumn{1}{c}{b} & \multicolumn{1}{c}{a} & \multicolumn{1}{c|}{c} \\  \cline{1-4}  \multicolumn{1}{|c}{a} & \multicolumn{1}{c}{c} & \multicolumn{1}{c}{a} & \multicolumn{1}{c|}{b} \\  \cline{1-4} \end{array} }
\case{1}{ \begin{array}{|c|c|c|} \cline{1-3}  \multicolumn{1}{|c}{a} & \multicolumn{1}{c}{b} & \multicolumn{1}{c|}{a} \\  \cline{1-3}  \multicolumn{1}{|c}{a} & \multicolumn{1}{c|}{b} \\  \cline{1-2} \end{array} }
\case{2}{ \begin{array}{|c|c|c|c|c|} \cline{2-5}  \multicolumn{1}{c|}{} & \multicolumn{1}{|c}{a} & \multicolumn{1}{c}{b} & \multicolumn{1}{c}{a} & \multicolumn{1}{c|}{c} \\  \cline{1-5}  \multicolumn{1}{|c}{a} & \multicolumn{1}{c}{c} & \multicolumn{1}{c}{a} & \multicolumn{1}{c|}{b} \\  \cline{1-4} \end{array} }
\case{3}{ \begin{array}{|c|c|c|} \cline{2-3}  \multicolumn{1}{c|}{} & \multicolumn{1}{|c}{a} & \multicolumn{1}{c|}{b} \\  \cline{1-3}  \multicolumn{1}{|c}{a} & \multicolumn{1}{c}{a} & \multicolumn{1}{c|}{b} \\  \cline{1-3} \end{array} }
\case{4}{ \begin{array}{|c|c|c|} \cline{1-3}  \multicolumn{1}{|c}{a} & \multicolumn{1}{c}{b} & \multicolumn{1}{c|}{a} \\  \cline{1-3}  \multicolumn{1}{|c}{a} & \multicolumn{1}{c}{a} & \multicolumn{1}{c|}{b} \\  \cline{1-3} \end{array} }
\case{5}{ \begin{array}{|c|c|c|c|} \cline{1-4}  \multicolumn{1}{|c}{a} & \multicolumn{1}{c}{b} & \multicolumn{1}{c}{a} & \multicolumn{1}{c|}{c} \\  \cline{1-4}  \multicolumn{1}{|c}{a} & \multicolumn{1}{c|}{b} \\  \cline{1-2} \end{array} }
\case{6}{ \begin{array}{|c|c|c|c|c|} \cline{3-5}  \multicolumn{2}{c|}{} & \multicolumn{1}{|c}{a} & \multicolumn{1}{c}{b} & \multicolumn{1}{c|}{a} \\  \cline{1-5}  \multicolumn{1}{|c}{a} & \multicolumn{1}{c}{c} & \multicolumn{1}{c}{a} & \multicolumn{1}{c|}{b} \\  \cline{1-4} \end{array} }
\case{7}{ \begin{array}{|c|c|c|c|c|} \cline{2-5}  \multicolumn{1}{c|}{} & \multicolumn{1}{|c}{a} & \multicolumn{1}{c}{b} & \multicolumn{1}{c}{a} & \multicolumn{1}{c|}{c} \\  \cline{1-5}  \multicolumn{1}{|c}{a} & \multicolumn{1}{c}{a} & \multicolumn{1}{c|}{b} \\  \cline{1-3} \end{array} }
\case{8}{ \begin{array}{|c|c|c|c|} \cline{3-4}  \multicolumn{2}{c|}{} & \multicolumn{1}{|c}{a} & \multicolumn{1}{c|}{b} \\  \cline{1-4}  \multicolumn{1}{|c}{a} & \multicolumn{1}{c}{c} & \multicolumn{1}{c}{a} & \multicolumn{1}{c|}{b} \\  \cline{1-4} \end{array} }
\case{9}{ \begin{array}{|c|c|c|c|} \cline{1-3}  \multicolumn{1}{|c}{a} & \multicolumn{1}{c}{b} & \multicolumn{1}{c|}{a} \\  \cline{1-4}  \multicolumn{1}{|c}{a} & \multicolumn{1}{c}{c} & \multicolumn{1}{c}{a} & \multicolumn{1}{c|}{b} \\  \cline{1-4} \end{array} }
\case{10}{ \begin{array}{|c|c|c|c|} \cline{1-4}  \multicolumn{1}{|c}{a} & \multicolumn{1}{c}{b} & \multicolumn{1}{c}{a} & \multicolumn{1}{c|}{c} \\  \cline{1-4}  \multicolumn{1}{c|}{} & \multicolumn{1}{|c}{a} & \multicolumn{1}{c}{a} & \multicolumn{1}{c|}{b} \\  \cline{2-4} \end{array} }
\case{11}{ \begin{array}{|c|c|} \cline{1-2}  \multicolumn{1}{|c}{a} & \multicolumn{1}{c|}{b} \\  \cline{1-2}  \multicolumn{1}{|c}{a} & \multicolumn{1}{c|}{b} \\  \cline{1-2} \end{array} }
\end{switch}}

\usepackage[pagewise,displaymath,mathlines]{lineno} 

\newtheorem{definition}{Definition}

\newtheorem{proposition}[definition]{Proposition}

\newtheorem*{conjecture*}{Conjecture}
\newtheorem{theorem}[definition]{Theorem}
\theoremstyle{remark}
\newtheorem{example}[definition]{Example}

\newcommand{\N}{\mathbb{N}}
\newcommand{\Z}{\mathbb{Z}}

\newcommand{\A}{\mathcal A}
\newcommand{\B}{\mathcal B}
\newcommand{\C}{\mathcal C}
\newcommand{\bu}{{\bf u}}

\newcommand{\bw}{{\bf w}}
\newcommand{\bz}{{\bf z}}
\renewcommand{\L}{\mathcal L}

\DeclareMathOperator\BA{B} 
\newcommand{\parikh}{\mathbf{1}}

\crefname{theorem}{theorem}{theorems}
\crefname{corollary}{corollary}{corollaries}
\crefname{example}{example}{examples}
\crefname{lemma}{lemma}{lemmas}
\crefname{proposition}{proposition}{propositions}
\crefname{definition}{definition}{definitions}
\crefname{observation}{observation}{observations}

\begin{document}

\title{On a conjecture about the absence of an initial balanced pair for Pisot substitutions}

\author{V\'{\i}ctor F. Sirvent%
\thanks{Electronic address: \texttt{vsirvent@usb.ve}}}
\affil{\footnotesize Departamento de Matem\'aticas, Universidad Sim\'on Bol\'{\i}var, Caracas, Venezuela}

\author{Štěpán Starosta%
\thanks{Electronic address: \texttt{stepan.starosta@fit.cvut.cz}}}
\affil{\footnotesize Faculty of Information Technology, Czech Technical University in Prague, Czech Republic}

\date{\today}

\maketitle

 \begin{abstract}
 \noindent
 Sellami and Sirvent~(\cite{sellami-sirvent}) conjectured that the balanced pair algorithm fails for the following pair of Pisot substitutions:
 \[
\varphi_0: \begin{array}{l}
a \mapsto abc\\
b \mapsto a\\
c \mapsto ac
\end{array}
\quad \text{ and } \quad
\varphi_1:
\begin{array}{l}
a \mapsto cba\\
b \mapsto a\\
c \mapsto ca
\end{array}.
\]

 The conjecture stated the balanced pair algorithm fails because there is no initial balanced pair.
 In the present note we prove this conjecture using a method based on simultaneous coding of the pair of the fixed points of the morphisms.
 \end{abstract}

\ifdraft{
   \listoftodos
}

\section{Introduction}

Morphisms and substitutions on words (over a finite alphabet) appear in different contexts of mathematics and theoretical computer science.
In particular in dynamical systems, the substitution dynamical systems are an important class
of symbolic dynamical systems. They have been studied extensively, for instance see~\cite{queffelec,CANT} and references therein.

We first recall some notions from combinatorics on words.
An {\em alphabet} $\A$ is a finite set and its elements are {\em letters}.
The elements of $\A^n$, for $n\geq 1$ are called {\em words on} $\A$, and $\A^0$ is the set formed by the empty word, $\varepsilon$.
Let $\A^*=\cup_{n\geq 0}\A^{n}$ be  the set of finite words on $\A$.
Let $\B$ be an alphabet.
A {\em morphism} $\varphi$ on $\A$ is a map $\varphi: \A \to \B^*$.
By setting $\varphi(uv)=\varphi(u)\varphi(v)$ for $u,v\in\A^*$, the morphism $\varphi$ is extended to $\A^*$.
Hence, the map $\varphi$ is extended in a straightforward manner to $\A^{\N}$, i.e. the set of one-sided infinite sequences over $\A$.
If $\varphi$ is an endomorphism, i.e. $\B=\A$, then it is also called a {\em substitution}.

An important question is to describe the common dynamics for two morphisms $\varphi_0$ and $\varphi_1$ that have the same incidence matrix. The {\em incidence matrix} of $\varphi$ is the matrix $M_{\varphi}=(m_{ij})$,
where $m_{ij}$ is the number of occurrences of the letter $i$ in the word $\varphi(j)$.
This problem has been addressed in~\cite{sirvent,sing-sirvent,sellami}.
The technique used in~\cite{sel0,sellami} is based on the balanced pair algorithm for two substitutions having the same incidence matrix.
This algorithm is a variation of the classical balanced pair algorithm introduced by Livshits~\cite{livshits} in the context of the Pisot conjecture, for more details see for instance~\cite{sirvent-solomyak}.

We say that a substitution is {\em Pisot} if the dominant   eigenvalue of the incidence matrix is a Pisot number.
In the case of Pisot substitutions, we can associate a geometrical object to the substitution, called Rauzy fractal~(\emph{cf.}~\cite{rauzy}).
Rauzy fractals play a fundamental role in the study of the Pisot conjecture.
The topological, geometrical and dynamical properties of the Rauzy fractals have been analyzed  extensively,  for instance see the recent survey~\cite{CANT} and references therein.
The study of the common dynamics of two Pisot morphisms is related to the intersection of their respective Rauzy fractals.
The balanced pair algorithm gives answers to this matter.

We introduce some notation that we use throughout the article:
For a word $u = u_0u_1 \cdots u_n \in \A^{n+1}$,
we set $u[i] = u_i$ for $0\leq i < n$.
For $\ell \leq n$, $u[{:}\ell] = u_0u_1\cdots u_{\ell-1} \in \A^\ell$.
The length of $u$ is denoted $|u|$, i.e., $|u| = n+1$.

Let $a \in \A$.
The number of occurrences of the letter $a$ in a finite word $u$ is denoted $|u|_a$.
We have $|u| = \sum_{a \in \A} |u|_a$.
Assuming that $\A$ is ordered, i.e., $\A = \{a_1,a_2,\ldots, a_k\}$, we set
\[
\parikh(u) = \left( |u|_{a_1} , |u|_{a_2}, \ldots, |u|_{a_k} \right) \in \N^k.
\]

Let $u$ and $v$ be two finite words, we say that $\left( u, v \right) $ is a \emph{balanced pair } if the number of occurrences of each letter of the alphabet in $u$ and $v$ are the same, i.e., $\parikh(u) = \parikh(v)$.
 We say that a \emph{balanced pair $\left( u, v \right) $ is minimal}  if
 $\parikh(u[{:}\ell])\neq\parikh(v[{:}\ell])$, for $1\leq\ell < |u|$.

Let $\varphi_0$ and $\varphi_1$ be two morphisms having the same incidence matrix.
 We assume that the morphisms are \emph{primitive}, i.e. there exists an integer $k$ such that all the entries of the matrix $M_{\varphi_0}^k$ are positive.
 Let $\bu_0$ and  $\bu_1$ be one-sided fixed points of $\varphi_0^m$ and $\varphi_1^m$, for some $m\geq 1$, i.e.
 $$
 \bu_0=\varphi_0^m(\bu_0), \quad \bu_1=\varphi_1^m(\bu_1).
 $$
 An \emph{initial balanced pair} for $(\bu_0,\bu_1)$ is a balanced pair $\left( u, v \right)$ such that $u$ and $v$ are prefixes of $\bu_0$ and $\bu_1$, respectively.

The \emph{balanced pair algorithm} for $(\bu_0,\bu_1)$ is defined as follows:
We assume that there exists an initial balanced pair $\left( u, v \right)$ for $\bu_0$ and $\bu_1$.
 We apply the morphisms $\varphi_0^m$, $\varphi_1^m$ as
 $\left( u, v \right) \mapsto \left( \varphi_0^m(u) , \varphi_1^m(v) \right)$.
 Since the morphisms have the same incidence matrix, the obtained pair $\left( \varphi_0^m(u) , \varphi_1^m(v) \right)$ is balanced.
 We decompose this balanced pair into minimal balanced pairs and we repeat this procedure for each new minimal balanced pair.
 Under the right hypothesis  the set of minimal balanced pairs is finite and the algorithm terminates (\emph{cf.}~\cite{sellami}).

The balanced pair algorithm requires that there is an initial balanced pair for the two sequences.
Sufficient conditions for Pisot morphisms to have an initial balanced pair are stated in~\cite{sellami}.
These conditions are in terms of geometrical properties of their respective Rauzy fractals.
A pair of substitutions is given in~\cite[Example 5]{sellami-sirvent} that does not satisfy these conditions.
It is conjectured that there is no initial balanced pair between the (one-sided) fixed points of the morphisms.
In the following theorem we give a positive answer to this conjecture.

\begin{theorem}\label{thm:main}
Let $\varphi_0$ and $\varphi_1$ be the morphisms over the alphabet $\left\{ a,b,c \right\}$ defined by
\begin{equation}\label{eqn:morphisms}
\varphi_0: \begin{array}{l}
a \mapsto abc\\
b \mapsto a\\
c \mapsto ac
\end{array}
\quad \text{ and } \quad
\varphi_1:
\begin{array}{l}
a \mapsto cba\\
b \mapsto a\\
c \mapsto ca
\end{array},
\end{equation}
and let $\bu_0$ and $\bu_1$ be their respective one-sided fixed points:
\begin{align*}
\bu_{0} = \varphi_0(\bu_0) = abcaacabcabcacabcaacabcaacabcacabc \ldots \\ 
\bu_{1} = \varphi_1(\bu_1)= cacbacaacbacacbacbacaacbacacbacaac \ldots 
\end{align*}
There is no initial balanced pair for $(\bu_0,\bu_1)$, i.e. for all $n>0$ we have
\begin{equation}
\parikh(\bu_0[{:}n]) \neq \parikh(\bu_1[{:}n]).
\end{equation}
\end{theorem}

The substitutions of Theorem~\ref{thm:main} are Pisot.
However we will not use any argument based on properties of Rauzy fractals or discuss any consequences related to the intersection of their respective Rauzy fractals.
The Rauzy fractals corresponding to these morphisms are depicted in~\cite[Figure 6]{sellami-sirvent}.
In the present note we shall use only a symbolic approach.

We give the proof of Theorem~\ref{thm:main} in \Cref{s:proof}.
The main idea for the proof   is the following.
We investigate the prefixes $\varphi_0(\bu_0[{:}n])$ and $\varphi_1(\bu_1[{:}n])$ at the same time.
For instance, for $n=1,2,3$ we have:
\begin{align*}
\varphi_0(\bu_0[{:}1]) &= abc, & \varphi_0(\bu_0[{:}2]) &= abca, & \varphi_0(\bu_0[{:}3]) &= abcaac, \\
\varphi_1(\bu_1[{:}1]) &= ca, & \varphi_1(\bu_1[{:}2]) &= cacba, & \varphi_1(\bu_1[{:}3]) &= cacbaca.
\end{align*}

To keep track of the letters at the same positions in the two fixed points and the images of letters at the same time, we encode the corresponding images of letters in the following manner:
\begin{align*}
\begin{array}{rl}
\varphi_0(\bu_0[{:}1]) & = \\
\varphi_1(\bu_1[{:}1]) & =
\end{array}
&
\letterOrig{0}
&
\begin{array}{rl}
\varphi_0(\bu_0[{:}2]) & = \\
\varphi_1(\bu_1[{:}2]) & =
\end{array}
&
\letterOrig{0} \
\letterOrig{1}
\\
\begin{array}{rl}
\varphi_0(\bu_0[{:}3]) & = \\
\varphi_1(\bu_1[{:}3]) & =
\end{array}
&
\letterOrig{0} \
\letterOrig{1} \
\letterOrig{2}
& &
\end{align*}
In that way, we construct a sequence which encodes $\bu_0$ and $\bu_1$ at the same time.

In the next section, we give the formal definitions and description of this construction.
\Cref{s:proof} contains a proof of \Cref{thm:main}.
In \Cref{s:comments} we discuss possible generalization of this technique of simultaneously generating and encoding  two fixed points  for a wider class of substitutions.

\section{Simultaneous coding of two fixed points} \label{sec:bricks}

In this section, we assume that $\varphi_0$ and $\varphi_1$ are two fixed substitutions of $\A$ and $\bu_0$ and $\bu_1$ are their two fixed points, respectively.
We start by giving a definition of the alphabet that shall be used for the simultaneous encoding.

\begin{definition} \label{de:brick}
Let $\A$ be an alphabet.
We set $\BA(\A) = \A \times \A \times \Z$ and call it the \emph{brick alphabet} (over $\A$).
An element of $\BA(\A)$ is called a \emph{brick}. 
The last element of a brick is called its \emph{offset}.
\end{definition}

\begin{definition}
Let $(v_0,v_1,s)$ and $(w_0,w_1,t)$ be two bricks.
We say that $(v_0,v_1,s)$ \emph{joins with} $(w_0,w_1,t)$ \emph{with respect to $(\varphi_0, \varphi_1)$} if
\[
| \varphi_0(v_0) | - s - | \varphi_1(v_1) | + t = 0.
\]
\end{definition}

In relation with the last definition and the morphisms $\varphi_0$ and $\varphi_1$, we shall depict bricks as in Introduction:
for a brick $(v_0,v_1,s)$ we place on the first row the word $\varphi_0(v_0)$ and on the second row we place $\varphi_1(v_1)$ with the offset $s$ relative to the first row.
One brick joins with another if the end of the first brick matches with the beginning of the second brick as demonstrated in the following example.

Consider $\varphi_0$ and $\varphi_1$ given in~(\ref{eqn:morphisms}).
Let $n = 1$.
Consider the brick $(a,c,0)$, depicted as
\[
\letterOrig{0},
\]
and the brick $(b,a,-1)$, depicted as
\[
\letterOrig{1}.
\]
The brick $(a,c,0)$ joins with $(b,a,-1)$ with respect to $(\varphi_0, \varphi_1)$ since
\[
|\varphi_0(a)| - 0 - |\varphi_1(c)| = 3 - 0 - 2 = 1.
\]
This fact can be visually depicted as
\[
\letterOrig{01}.
\]
The brick $(b,a,-1)$ joins with $(c,c,1)$ and we have
\[
\letterOrig{012}.
\]
The brick $(b,a,-1)$ does not join with itself:
\[
\letterOrig{1} \ \letterOrig{1}.
\]

Let us introduce a notation that allows us to retrieve the encoded sequences from sequences over a brick alphabet $\BA(\A)$.
Let $z = \left( \left( w_{0,i},w_{1,i}, k_i \right)  \right)_{i}  $ be an element of $\BA(\A)^{\N}$.
For $j\in\{0,1\}$, we set $\tau_j: \BA(\A)^\N \to \A^\N$ by $\tau_j(z) = w_{j,0} w_{j,1} w_{j,2}\ldots\in\A^{\N}$.

For any $\varphi_0$ and $\varphi_1$, we can find a sequence $\bz$ over $\BA(\A)$ in the spirit of the above example.
We construct the bricks to respect the original sequences: $\tau_i(\bz) = \bu_i$ for $i \in \{0,1\}$.
The offsets of the bricks in $\bz$ are set position by position starting from the beginning.
The first offset is $0$ and the offsets of other bricks in $\bz$ is set so that any brick in $\bz$ joins (with respect to $(\varphi_0,\varphi_1)$) with the brick immediately following it.
We say that $\bz$ \emph{simultaneously encodes} the fixed points $\bu_0$ and $\bu_1$.

Continuing the above example, we find that the sequence simultaneously encoding the fixed points of~(\ref{eqn:morphisms}) starts with
\[
\left( a,c, 0 \right) \left( b,a,-1 \right) \left( c,c,1 \right)  \ldots
\]

\section{Proof of Theorem~\ref{thm:main}}\label{s:proof}
In this section, $\varphi_0$ and $\varphi_1$ are fixed to be the two morphisms in~(\ref{eqn:morphisms}), and
 $\bu_0$ and $\bu_1$ are their respective fixed points.
The alphabet considered is $\A = \{a,b,c\}$.

Let $\mu$ be a substitution over $\C = \{0,1, \ldots, 8 \}$ given by
\[
\mu:
\begin{array}{lllll}
0 \mapsto 01, &
1 \mapsto 23, &
2 \mapsto 45, &
3 \mapsto 41, &
4 \mapsto 231, \\
5 \mapsto 26, &
6 \mapsto 478, &
7 \mapsto 2, &
8 \mapsto 66.
\end{array}
\]
Let $\bw$ denote its fixed point.
We have
\[
\bw = 012345412312623123454123454784541234541231 \ldots
\]

Inspecting $\mu$, we find the set of all its factors of length 2 easily:
\[
\L_2(\bw) =  \left\{ 01, 66, 12, 31, 23, 34, 26, 45, 62, 47, 84, 78, 54, 41 \right\}.
\]

Let $\B = \{ (a,c,0),
(b,a,-1),
(c,c,1),
(a,b,1),
(a,a,-1),
(c,c,-1),
(a,a,1),
(b,c,-1),
(c,b,0)
\} \subset \BA(\A)$
and $\pi: \C^* \to \B^*$ be the morphism determined by
\[
\pi:
\begin{array}{rlrlrlrlrl}
0 & \mapsto (a,c,0), &
1 & \mapsto (b,a,-1), &
2 & \mapsto (c,c,1), &
3 & \mapsto (a,b,1), &
4 & \mapsto (a,a,-1), \\
5 & \mapsto (c,c,-1), &
6 & \mapsto (a,a,1), &
7 & \mapsto (b,c,-1), &
8 & \mapsto (c,b,0).
\end{array}
\]

The action of $\pi$ may be depicted with respect to $(\varphi_0,\varphi_1)$ as follows:
\begin{equation} \label{eq:orig_pi_bricks}
\pi: \begin{array}{rlrlrl}
0 & \mapsto  \letterOrig{0}, &
1 & \mapsto  \letterOrig{1}, &
2 & \mapsto  \letterOrig{2},\\[2em]
3 & \mapsto  \letterOrig{3}, &
4 & \mapsto  \letterOrig{4}, &
5 & \mapsto  \letterOrig{5}, \\[2em]
6 & \mapsto  \letterOrig{6}, &
7 & \mapsto  \letterOrig{7}, &
8 & \mapsto  \letterOrig{8}.\\[2em]
\end{array}
\end{equation}

The sequence $\bw$ is constructed so that the sequence $\pi(\bw) = \pi(\bw[0])\pi(\bw[1])\pi(\bw[2])\ldots$ simultaneously encodes both $\bu_0$ and $\bu_1$.
This fact is shown by the two following propositions.
The alphabet $\C$ and the morphism $\pi$ serve only to work with the indices of the elements of $\B$ to ease the presentation. 

\begin{proposition} \label{st:orig_u_are_produced}
Let $i \in \{0,1\}$.
We have $\bu_i = \tau_i \pi (\bw)$.
\end{proposition}

\begin{proof}
We show that for all $n$, the word $\tau_i \pi  \mu ( \bw[{:}n] ) $ is a prefix of $\varphi_i \tau_i \pi ( \bw[{:}n] )$ (which is a prefix of $\bu_i$).
The proof will be done by induction on $n$.

In fact, we will prove a stronger statement:
there exists a mapping $s_i: \C \to \A^*$ such that for all $n$ we have
\begin{equation}
\tau_i \pi \mu ( \bw[{:}(n{+}1)] ) s_i(x) = \varphi_i \tau_i \pi ( \bw[{:}(n{+}1)] ) \label{eq:indu_x}
\end{equation}
where $x$ is the last letter of $\bw[{:}(n{+}1)]$, i.e., $\bw[{:}(n{+}1)] = \bw[{:}n]x$.

We set $s_i$ as follows:
\begin{align*}
s_0: & & 0 \mapsto c, 1 \mapsto \varepsilon, 2 \mapsto \varepsilon, 3 \mapsto c, 4 \mapsto c, 5 \mapsto c, 6 \mapsto \varepsilon, 7 \mapsto a, 8  \mapsto c, \\
s_1: & & 0 \mapsto \varepsilon, 1 \mapsto a, 2 \mapsto a, 3 \mapsto \varepsilon, 4 \mapsto \varepsilon, 5 \mapsto \varepsilon, 6 \mapsto a, 7 \mapsto a, 8  \mapsto \varepsilon.
\end{align*}
Taking all $yx \in \L_2(\bw) =  \left\{ 01, 66, 12, 31, 23, 34, 26, 45, 62, 47, 84, 78, 54, 41 \right\}$, we may verify that
\begin{equation} \label{eq:indux_ro}
\tau_i \pi \mu (x) s_i(x) = s_i(y) \varphi_i \tau_i \pi (x).
\end{equation}

To show \eqref{eq:indu_x} we proceed by induction on $n$.
The claim is verified for $n = 0$ ($\bw[{:}1] = \bw[0] = 0$):
\begin{align*}
\tau_0 \pi \mu ( 0 ) s_0(0) & = abc = \varphi_0 \tau_0 \pi (0), \\
\tau_1 \pi \mu ( 0 ) s_1(0) & = ca = \varphi_1 \tau_1 \pi (0).
\end{align*}
Assume \eqref{eq:indu_x} holds for $n \geq 0$.
Let $\bw[{:}(n{+}2)] = \bw[{:}n]yx$, i.e., $y$ and $x$ be the last two letters of $\bw[{:}(n{+}2)]$.
We have
\[
\tau_i \pi \mu ( \bw[{:}(n{+}2)] ) s_i(x) = \tau_i \pi \mu ( \bw[{:}(n{+}1)] ) \tau_i \pi \mu (x) s_i(x) = \tau_i \pi \mu ( \bw[{:}(n{+}1)] ) s_i(y) \varphi_i \tau_i \pi (x)
\]
where the last equality follows from \eqref{eq:indux_ro}.
Since by using the induction assumption we have
\[
\tau_i \pi \mu ( \bw[{:}(n{+}1)] ) s_i(y) \varphi_i \tau_i \pi (x) = \varphi_i \tau_i \pi ( \bw[{:}(n{+}1)] ) \varphi_i \tau_i \pi (x) =
\varphi_i \tau_i \pi ( \bw[{:}(n{+}2)] ),
\]
we conclude that
\[
\tau_i \pi \mu ( \bw[{:}(n{+}2)] ) s_i(x) = \varphi_i \tau_i \pi ( \bw[{:}(n{+}2)] ),
\]
which finishes the induction step and \eqref{eq:indu_x} is proven.
\end{proof}

The following claim states that all two consecutive bricks of $\pi(\bw)$ join one with another.

\begin{proposition} \label{st:orig_w_compatible}
If $b_0b_1 \in \L_2(\pi(\bw))$, then $b_0$ joins with $b_1$ with respect to $(\varphi_0,\varphi_1)$.
\end{proposition}

\begin{proof}
Checking each $b_0b_1 \in \L_2(\pi(\bw))$, we find that $b_0$ joins with $b_1$ with respect to $(\varphi_0,\varphi_1)$.
Indeed, using the graphical view \eqref{eq:orig_pi_bricks} we obtain
\begin{equation} \label{eq:orig_piL2}
\begin{aligned}
\pi(01) & = \letterOrig{01} &
\pi(66) & = \letterOrig{66} &
\pi(12) & = \letterOrig{12} \\
\pi(31) & = \letterOrig{31} &
\pi(23) & = \letterOrig{23} &
\pi(34) & = \letterOrig{34} \\
\pi(26) & = \letterOrig{26} &
\pi(45) & = \letterOrig{45} &
\pi(62) & = \letterOrig{62} \\
\pi(47) & = \letterOrig{47} &
\pi(84) & = \letterOrig{84} &
\pi(78) & = \letterOrig{78} \\
\pi(54) & = \letterOrig{54} &
\pi(41) & = \letterOrig{41} &
\end{aligned}
\end{equation}
and the proof is concluded by checking that there are no gaps or overlaps between the two bricks depicted with respect to $(\varphi_0,\varphi_1)$.
\end{proof}

We are set for proving Theorem~\ref{thm:main}.

\begin{proof}[Proof of Theorem~\ref{thm:main}]
We will prove that for all $n > 0$ we have $\left |\bu_0[{:}n] \right |_c < \left |\bu_1[{:}n] \right |_c$.

We shall depict  $\pi(\bw)$ with respect to $\varphi_0$ and $\varphi_1$, i.e., using \eqref{eq:orig_pi_bricks}.
We obtain
\[
\pi(\bw) =
\letterOrig{0123454}
\ldots
\]
By \Cref{st:orig_u_are_produced,st:orig_w_compatible} we conclude that on the first row of the right side of the previous equation we have exactly $\bu_0$ (with no spaces between letters or overlaps), and that the second row is exactly $\bu_1$.
Moreover, since the first brick is $(a,c,0)$ with offset $0$, the words $\bu_0$ and $\bu_1$ are aligned and their positions match.
Comparing all occurrences of the letter $c$ in $\pi(\L_2(\bw))$ depicted in \eqref{eq:orig_piL2} we conclude that almost all occurrences of $c$ in $\bu_0$ may be paired with occurrences of $c$ in $\bu_1$.
Precisely, for $n > 0$ we have
\begin{align*}
\text{if }\bu_0[n] = c \quad \text{then} \quad \bu_1[n{-}1]\bu_1[n] \text{contains exactly 1 occurrence of } c, \\
\text{if }\bu_1[n] = c \quad \text{then} \quad \bu_0[n]\bu_0[n{+}1] \text{contains exactly 1 occurrence of } c.
\end{align*}
The first letter of $\bu_1$ is $c$ and it is not paired with any $c$ in $\bu_0$.
The relative positions of $c$'s in $\bu_0$ and $\bu_1$ imply
\[
\left | \bu_1[{:}n] \right |_c > \left | \bu_0[{:}n] \right |_c
\]
for all $n > 0$.
\end{proof}

\section{Comments} \label{s:comments}

In this section we discuss the used technique to simultaneously generate two fixed points in a more general view.
This technique was introduced in this article in order to check that there is no initial balanced pair for a pair of fixed points; however the technique is interesting by itself and it could have other potential applications.
Clearly, a crucial step of the technique is to have a simultaneous coding which is a fixed point of a morphism.

We explore here if the same method may be used to answer the same question for some other pair of morphisms.
Since we only intend to make comments and state open questions, we don't give any proofs as they would in the very same spirit as above and we support our claims by computer evidence.

We give two simple examples while keeping the same meaning of the notation:
The sequence $\pi(\bw)$ is the sequence simultaneously encoding the two fixed points and the goal is to find a fixing morphism $\mu$ of $\bw$.
Again, $\C$ and $\pi$ are used only to ease the presentation.

\begin{example} \label{ex:1}
We consider the following two morphisms over $\A = \{a,b\}$:
\[
\varphi_0: \begin{array}{l}
a \mapsto aab\\
b \mapsto ab
\end{array}
\quad \text{ and } \quad
\varphi_1:
\begin{array}{l}
a \mapsto aba\\
b \mapsto ba
\end{array}.
\]

We set $\C=\{0,1,2\}$, $\B = \left\{ (a,b,0),(a,a,-1),(b,a,-1) \right\}$ and $\pi: \C^* \to \B^*$ depicted with respect to $\left( \varphi_0, \varphi_1 \right) $ as follows
\[
\pi: 0 \mapsto \begin{array}{|c|c|c|} \cline{1-3}  \multicolumn{1}{|c}{a} & \multicolumn{1}{c}{a} & \multicolumn{1}{c|}{b} \\  \cline{1-3}  \multicolumn{1}{|c}{b} & \multicolumn{1}{c|}{a} \\  \cline{1-2} \end{array} , \quad
1 \mapsto \begin{array}{|c|c|c|c|} \cline{2-4}  \multicolumn{1}{c|}{} & \multicolumn{1}{|c}{a} & \multicolumn{1}{c}{a} & \multicolumn{1}{c|}{b} \\  \cline{1-4}  \multicolumn{1}{|c}{a} & \multicolumn{1}{c}{b} & \multicolumn{1}{c|}{a} \\  \cline{1-3} \end{array} , \quad
2 \mapsto \begin{array}{|c|c|c|} \cline{2-3}  \multicolumn{1}{c|}{} & \multicolumn{1}{|c}{a} & \multicolumn{1}{c|}{b} \\  \cline{1-3}  \multicolumn{1}{|c}{a} & \multicolumn{1}{c}{b} & \multicolumn{1}{c|}{a} \\  \cline{1-3} \end{array}.
\]
Let $\bw \in \C^\N$ be fixed by the morphism determined by $0 \mapsto 01, 1 \mapsto 201, 2 \mapsto 202$.
The sequence $\pi(\bw)$ simultaneously encodes $\bu_0$ and $\bu_1$.
\end{example}

\begin{example} \label{ex:2}
We consider $\A = \{a,b,c\}$ and the substitutions
\[
\varphi_0: \begin{array}{l}
a \mapsto abac\\
b \mapsto aba\\
c \mapsto ab
\end{array}
\quad \text{ and } \quad
\varphi_1:
\begin{array}{l}
a \mapsto acab\\
b \mapsto aab\\
c \mapsto ab
\end{array}.
\]

We set $\C = \{0,1,\ldots,11\}$ and
\[
\B = \left\{\begin{array}{l}
 (a,a,0),
(b,c,0),
(a,a,-1),
(c,b,-1),
(b,b,0),
(a,c,0),\\
(b,a,-2),
(a,b,-1),
(c,a,-2),
(b,a,0),
(a,b,1),
(c,c,0)
\end{array}
\right\}.
\]

The morphism $\pi: \C^* \to \B^*$ maps the elements of $\C$ to the elements of $\B$ in the order as depicted above, its graphical view with respect to $(\varphi_0,\varphi_1)$ is as follows:
\[
\begin{aligned}
\begin{aligned}
\pi(0) & = \letterzuj{0}, &
\pi(1) & = \letterzuj{1}, &
\pi(2) & = \letterzuj{2}, \\
\pi(3) & = \letterzuj{4}, &
\pi(4) & = \letterzuj{5}, &
\pi(5) & = \letterzuj{6}, \\
\pi(6) & = \letterzuj{6}, &
\pi(7) & = \letterzuj{8}, &
\pi(8) & = \letterzuj{8}, \\
\pi(9) & = \letterzuj{9}, &
\pi(10) & = \letterzuj{10}, &
\pi(11) & = \letterzuj{11}.\\
\end{aligned}
\end{aligned}
\]

The morphism fixing $\bw$ is determined by
\[
\begin{aligned}
0 & \mapsto 0123, &
1 & \mapsto 0, &
2 & \mapsto 405678,\\
3 & \mapsto 04, &
4 & \mapsto 0.9.10, &
5 & \mapsto 0.4.0.11.0.4,\\
6 & \mapsto 0, &
7 & \mapsto 0.4.0.11, &
8 & \mapsto 04\\
9 & \mapsto 0, &
10 & \mapsto 127623, &
11 & \mapsto 04.\\
\end{aligned}
\]
Since the alphabet has more than $10$ letters and we use decimal numbers to represent them, if needed, we use ``.'' to separate the elements of a word in order to avoid confusions.
\end{example}

We would like to point out that \Cref{ex:1,ex:2} trivially have  an initial balanced pair, i.e. $(a,a)$.

In general, this direct approach to simultaneously generate two fixed points does not work for any pair of morphisms.
We do not consider the computational problems, i.e., when the size of $\B$ is large.
The combinatorial problem is that $\bw$ may not be a fixed point.
We continue with examples that displayed this property and discuss possible refinements of the method.

\begin{example}
The pair of morphisms~(\ref{eqn:morphisms})  belongs to the following family
\begin{equation} \label{eqn:morphisms_k}
\varphi_{k,0}: \begin{array}{l}
a \mapsto a^kbc\\
b \mapsto a\\
c \mapsto ac
\end{array}
\quad \text{ and } \quad
\varphi_{k,1}:
\begin{array}{l}
a \mapsto cba^k\\
b \mapsto a\\
c \mapsto ca
\end{array},
\end{equation}
where $a^k=\underbrace{a\cdots a}_{k-\text{times}}$, with $k\geq 1$.

Computer experiments indicate that their respective fixed points do not have an initial balanced pair.
The methods expounded in this note do not work well for $k>1$.
\end{example}

The first refinement is to alternate \Cref{de:brick}.

\begin{definition}[Extended~\Cref{de:brick}]
Let $\A$ be an alphabet and $n$ is a positive integer.
We set $\BA_n(\A) = \A^n \times \A^n \times \Z$ and call it the \emph{brick alphabet of order $n$} (over $\A$).
An element of $\BA_n(\A)$ is called a \emph{brick of order $n$}.  
\end{definition}

Using this notion, the problem for the morphisms~\eqref{eqn:morphisms_k} with $k=2$ is remedied by taking bricks of order $2$.
(The simultaneous encoding using bricks of order $2$ is done so that in the graphical view the bricks overlap by a brick of order $1$, i.e., $\tau_i(\bw) = \bu_i[0]\bu_i[1], \bu_i[1]\bu_i[2], \bu_i[2]\bu_i[3] \ldots \in \left(   \L_2(\bu_i) \right)^\N$.)
We find $\#\B = \#\C = 21$ and $\bw$ is fixed by a morphism.
However, for $k=3$, computer evidence suggests that the sequence $\bw$ is not fixed by any morphism.

In order to generate $\bw$ one might try to check whether it can be generated as a morphic image of a fixed point (so-called \emph{morphic sequence}).
This is a second refinement of the exposed technique.
An intuitive approach is to select a factor of $\bw$, construct a coding of return words to this factor in $\bw$ (usually called a \emph{derivated sequence}, see \cite{Durand98}) or a coding of return words to a selected set of factors, and test whether this coding is a fixed point.
Using these two refinements, we have successfully solved the problem for some random examples.
However, a sufficient condition for any of the techniques to work is an open question.
We state this question along with other open problems:

\begin{enumerate}
\item Primitivity of $\varphi_0$ (and thus of $\varphi_1$) is a sufficient condition for $\big |   \left | \varphi_0(\bu_0[{:}n]) \right | - \left | \varphi_1(\bu_1[{:}n]) \right |  \big |$ to be bounded.
It follows that the simultaneous coding of $\bu_0$ and $\bu_1$ is over a finite alphabet since the offsets of the bricks occurring in it are bounded.
It implies that using bricks of higher orders will also yield a finite alphabet.

To see that in the non-primitive case we may need an infinite number of bricks consider the following example:
\[
\varphi_0: \begin{array}{l}
a \mapsto abacaa\\
b \mapsto cbb\\
c \mapsto bcc
\end{array}
\quad \text{ and } \quad
\varphi_1:
\begin{array}{l}
a \mapsto baacaa\\
b \mapsto bbc\\
c \mapsto bcc
\end{array} .
\]
The incidence matrix of these two morphisms is $ \begin{pmatrix}
4 & 0 & 0 \\
1 & 2 & 1 \\
1 & 1 & 2
\end{pmatrix}$, its largest eigenvalue is $4$ and the associated eigenvector is $(1,1,1)^T$.
The other eigenvalues are $3$ and $1$ with respective eigenvectors $(0,1,1)^T$ and $(0,1,-1)$.
It follows that the order of growth of $\left | \varphi_0^n ( a ) \right |$ is $4^n$,
while the order of growth of $\left | \varphi_1^n ( b ) \right |$ is $3^n$.
Therefore, the simultaneous coding of the fixed points of these two morphisms requires an infinite number of bricks.
Clearly, this property is due to the choice of distinct first letters of the compared fixed points $a$ for $\varphi_0$ and $b$ for $\varphi_1$ which have distinct magnitude of growth.

A further investigation of conditions for the finiteness of the number of bricks needed is an open questions along with an estimate on this number.
\item Assuming the number of required bricks for a simultaneous coding is finite, one can be interested in an integer $N$ such that all required bricks can be found using the words $\varphi_0^N(\bu_0[0])$ and $\varphi_1^N(\bu_1[0])$.
\item What are some sufficient conditions so that a simultaneous coding $\bw$ (using bricks of order $h$) is a fixed point or a morphic image of a fixed point?
\item There is an overlap of our presented algorithm with the approach used in~\cite{sellami} based on the balanced pair algorithm, recalled in Introduction.
This approach identifies minimal balanced pairs which can be seen as sequences of various lengths of bricks of order $1$, always ending with a brick with offset $0$, or as bricks of various orders with offset $0$.
 A description of the relation of the two algorithms is an open question.
For instance, if one terminates, does the other also terminate?

In~\cite{sellami}, the author investigates specific pairs of substitutions (unimodular irreductible Pisot substitutions satisfying the Pisot conjecture and having $0$ in the interior of their Rauzy fractal, see~\cite{sellami,Aper}) for which the approach based on minimal balanced pairs terminates.
In Section 5 of~\cite{sellami}, there are the two following examples.
The first one consists of morphisms determined by $a \mapsto aba, b \mapsto ab$ and $a \mapsto aab, b \mapsto ba$.
Our algorithm terminates by finding a morphism fixing the simultaneous coding in question, which is over $6$ bricks of order $1$.
The second example is formed by morphisms determined by $a \mapsto ab,b \mapsto ac,c \mapsto a$ and $a \mapsto ab,b \mapsto ca,c \mapsto a$.
In this case, using bricks of order $1$, the simultaneous coding seems not to be fixed by a morphism.
This observation supports the need to explore this question further.
\end{enumerate}

\section*{Acknowledgements}

ŠS acknowledges financial support by the Czech Science Foundation grant GA\v CR 13-03538S.
Computer experiments were performed using SageMath~\cite{sage}.


\end{document}